\documentclass{amsart}

\usepackage{amsmath,amsthm,amsfonts,amssymb,bm,enumitem,graphicx,xparse}

\makeatletter
\@namedef{subjclassname@2020}{%
  \textup{2020} Mathematics Subject Classification}
\makeatother

\makeatletter
\NewDocumentCommand{\sump}{e{_}}
 {%
  \DOTSB
  \mathop{\IfNoValueTF{#1}{\sump@{}}{\sump@{#1}}}%
  \nolimits
 }
\newcommand{\sump@}[1]{\mathpalette\sump@@{#1}}
\newcommand{\sump@@}[2]{%
  \ifx#1\displaystyle
    {\sump@display{#2}}%
  \else
    \sum@\nolimits^*_{#2}%
  \fi
}
\newcommand{\sump@display}[1]{%
  \sbox\z@{$\m@th\displaystyle\sum@\nolimits'$}%
  \sbox\tw@{$\m@th\displaystyle\sum@\limits_{#1}$}%
  \sbox\@tempboxa{$\m@th\displaystyle'$}
  \mathop{\sum@\nolimits^* \kern-\wd\@tempboxa}\limits_{#1}%
  \ifdim\wd\z@>\wd\tw@
    \kern\dimexpr\wd\z@-\wd\tw@\relax
  \fi
}
\makeatother

\addtocontents{toc}{\setcounter{tocdepth}{1}} 

\usepackage%[backref=page]
{hyperref}
\hypersetup{colorlinks=true,citecolor=blue,linkcolor=blue,urlcolor=blue,
pdfstartview=FitH }

\theoremstyle{plain}
\newtheorem{theorem}{Theorem}

\newtheorem{lemma}[theorem]{Lemma}
\newtheorem{prop}[theorem]{Proposition}

\theoremstyle{definition}
\newtheorem*{Acknowledgements}{Acknowledgements}
\newtheorem*{remark}{Remark}
\newtheorem*{notation}{Notation}

\renewcommand{\geq}{\geqslant}
\renewcommand{\leq}{\leqslant}
\renewcommand{\mod}{\mathrm{mod}\,}

\newcommand{\res}{\mathrm{res}\,}

\newcommand{\pri}{\mathfrak{p}}
\newcommand{\Pri}{\mathfrak{P}}
\newcommand{\RR}{\mathbb{R}}
\newcommand{\CC}{\mathbb{C}}
\newcommand{\QQ}{\mathbb{Q}}
\newcommand{\ZZ}{\mathbb{Z}}
\newcommand{\NN}{\mathbb{N}}

\numberwithin{equation}{section}

\makeatletter
\def\Ddots{\mathinner{\mkern1mu\raise\p@
\vbox{\kern7\p@\hbox{.}}\mkern2mu
\raise4\p@\hbox{.}\mkern2mu\raise7\p@\hbox{.}\mkern1mu}}
\makeatother
\makeatletter
\DeclareRobustCommand\widecheck[1]{{\mathpalette\@widecheck{#1}}}
\def\@widecheck#1#2{%
    \setbox\z@\hbox{\m@th$#1#2$}%
    \setbox\tw@\hbox{\m@th$#1%
       \widehat{%
          \vrule\@width\z@\@height\ht\z@
          \vrule\@height\z@\@width\wd\z@}$}%
    \dp\tw@-\ht\z@
    \@tempdima\ht\z@ \advance\@tempdima2\ht\tw@ \divide\@tempdima\thr@@
    \setbox\tw@\hbox{%
       \raise\@tempdima\hbox{\scalebox{1}[-1]{\lower\@tempdima\box
\tw@}}}%
    {\ooalign{\box\tw@ \cr \box\z@}}}
\makeatother

\begin{document} 

\title{Density of the union of positive diagonal binary quadratic forms}

\author{Yijie Diao}

\address{Institute of Science and Technology Austria, Am Campus 1, 3400 Klosterneuburg, Austria}
\email{s6yidiao@gmail.com}

\date{\today}

\begin{abstract} 
Let $X$ be a sufficiently large positive integer. We prove that one may choose a subset $S$ of primes with cardinality $O(\log X)$, such that a positive proportion of integers less than $X$ can be represented by $x^2 + py^2$ for at least one of $p \in S$. 
\end{abstract}


\setcounter{tocdepth}{2}  \maketitle

\section{Introduction}\label{intro}

A celebrated theorem of Landau \cite{La} states that
\begin{equation*}
	\# \{ n \leq X: n = x^2 + y^2, \ x,y \in \mathbb{Z} \} \sim \kappa X (\log X)^{-1/2},
\end{equation*} 
where $\kappa$ is the Landau-Ramanujan constant. It implies that the density of the integers which can be represented by the sum of two squares is zero.
Bernays \cite{Be} generalized Landau's result by showing that
\begin{equation}\label{Bernays}
	\# \{ n \leq X: n = x^2 + z y^2, \ x,y \in \mathbb{Z} \} \sim \kappa_z X (\log X)^{-1/2},
\end{equation}
for any positive integer $z$.

In this paper, we consider the set
\begin{equation*}
	N(X, Z) = \ \# \bigcup_{1 \leq z \leq Z} \{ n \leq X : n = x^2 + z y^2 \}.
\end{equation*} 
It is natural to ask when $Z$ is sufficiently large, with respect to $X$, such  that $N(X, Z)$ takes a positive proportion of $X$.
According to (\ref{Bernays}), we require  $Z \gg (\log X)^{1/2}$ at least. 
In an interesting and difficult recent paper, Hanson and Vaughan \cite{HV} showed that $Z \gg  \log X ( \log\log X )$ is sufficient.
More precisely, they chose a subset $S'(Z) = \{z \leq Z: z \text{ is odd and square-free}\}$ and proved for $Z = \log X (\log \log X)$ that
\begin{equation*}
	\# \bigcup_{z \in S'(Z)} \{ n \leq X : n = x^2 + z y^2 \} \gg X.
\end{equation*}

In this paper, we will use a very different approach from \cite{HV} to give a simpler proof of a stronger result. 
Note that $\#S'(Z) \asymp Z = \log X ( \log\log X )$. 
We will show that one can choose a strictly smaller subset with size $O(\log X)$. 

For a subset $\Omega \subset \NN$, we denote
\begin{equation*}
	S_\Omega(Z) = \{ z \leq Z : z \in \Omega \},
\end{equation*}
and
\begin{equation*}
	N_\Omega(X, Z) = \ \# \bigcup_{z \in S_\Omega(Z)} \{ n \leq X : n = x^2 + z y^2 \}.
\end{equation*} 
Let $\mathcal{P}$ be the set of all primes congruent to 1 modulo 4 and let $\mathcal{Q}$ be a subset of $\mathcal{P}$ satisfying
\begin{equation*}
	\# S_\mathcal{Q}(Z) \asymp Z / \log Z,
\end{equation*}
for $Z$ sufficiently large. 
Our primary result is the following theorem.
\begin{theorem}\label{maintheorem}
	For $Z =  \log X ( \log\log X )$ and $\mathcal{Q}$ as above, we have
	$$N_\mathcal{Q}(X, Z) \gg X.$$
\end{theorem} 
The implicit constant only depends on the subset $\mathcal{Q} \subset \mathcal{P}$. Note that with the choice of $Z$, we have $\# S_\mathcal{Q}(Z) \asymp \log X$.

\subsection{Preliminaries}
Following the notation of \cite{BG} and \cite{HV}, we call two solutions $(x_1, y_1)$ and $(x_2, y_2)$ with $x_1^2 + zy_1^2 = x_2^2 + zy_2^2$ equivalent if there is an automorphism $A \in \mathrm{SL}_2(\ZZ)$ such that $(x_1, y_1) = (x_2, y_2) A$. 
By considering the field norm, we know that the number of automorphisms $g_z$ equals the number of units in the ring of integers of $\QQ(\sqrt{-z})$.\footnote{ If we restrict $z$ to be square-free and $z \equiv 1 \, (\mod 4)$, then $g_z$ is always equal to $2$. } 
We denote
\begin{equation*}
	r(n, z) = \# \{ (x,y) \in \ZZ^2: x^2 + zy^2 = n \} / g_z,
\end{equation*}	 
and
\begin{equation*}
	R_\mathcal{Q}(n, Z) = \sum_{p \in S_\mathcal{Q}(Z)} p^{1/2} \, r(n, p).
\end{equation*}

\vspace{0.25em}
The Cauchy-Schwarz inequality shows that
\begin{equation}\label{C-S}
	N_\mathcal{Q}(X, Z) \cdot \sum_{n \leq X} R_\mathcal{Q}(n, Z)^2 \geq \Big( \sum_{n \leq X} R_\mathcal{Q}(n, Z) \Big)^2.
\end{equation}
For the right hand side, according to \cite[Lemma 3.1]{BG} we have uniformly
\begin{equation*}
	\sum_{n \leq X} r(n, z) = \pi z^{-1/2} X + O(X^{1/2}).
\end{equation*}
Then for $Z \ll X^{1/10}$, we have
\begin{equation*}
	\sum_{n \leq X} R_\mathcal{Q}(n, Z) = \sum_{p \in S_\mathcal{Q}(Z)} \Big( p^{1/2} \cdot \sum_{n \leq X} r(n, p) \Big)
	\asymp X (Z / \log Z) .
\end{equation*} 
Hence by (\ref{C-S}), we need to find when $Z$ is sufficiently large, such that
\begin{equation}\label{R2}
	\sum_{n \leq X} R_\mathcal{Q}(n, Z)^2 \ll X (Z / \log Z)^2.
\end{equation}

For the estimate of $\sum_{n} R_\mathcal{Q}(n,Z)^2$, we treat the diagonal part and the off-diagonal part separately, namely
\begin{equation*}
	R_\mathcal{Q}(n, Z)^2 =  \sum_{p \in S_\mathcal{Q}(Z)} p \, r(n, p)^2 
	+ \sum_{\substack{p_1 \neq p_2 \\ p_1, \, p_2 \in S_\mathcal{Q}(Z)}} (p_1 p_2)^{1/2} r(n, p_1) r(n, p_2).
\end{equation*}
By \cite[Theorem 2]{BG} we have uniformly
\begin{equation*}
	\sum_{n \leq X} r(n, z)^2 = 2\pi z^{-1/2} X + E_z(X),
\end{equation*}
where the error term $E_z(X)$ satisfies the following bound:
\begin{equation*}
	E_z(X) \ll X^{1/2} + \tau(z) \big( X \log X \cdot z^{-1} + X \cdot z^{-3/4} \big).
\end{equation*}
Hence for $Z \ll X^{1/10}$ and $p \in S_\mathcal{Q}(Z)$, we have

\begin{equation*}
	\sum_{n \leq X} r(n, p)^2 \ll X (p^{-1/2} + \log X \cdot p^{-1} ).
\end{equation*}
We have the following estimate for the diagonal part:
\begin{equation}\label{diag}
    \sum_{n \leq X} \sum_{p \in S_\mathcal{Q}(Z)} p \, r(n, p)^2
	   \ll X \cdot (Z / \log Z) \cdot (Z^{1/2} + \log X).
\end{equation}

Therefore, in order to obtain (\ref{R2}), we will choose
$$ Z \gg \log X (\log \log X).$$
This reveals the limitation of the Cauchy-Schwarz inequality (\ref{C-S}). 

\vspace{0.1em}
By positivity, we may and will from now on only consider the full set $\mathcal{P}$.
It is sufficient to prove Theorem \ref{maintheorem} by showing the following proposition.
\begin{prop}\label{positive}
For $Z = \log X (\log \log X)$ and $\mathcal{P} = \{p \text{ prime}: p \equiv 1 \, (\mod 4)\}$, we have
	\begin{equation*}
	\sum_{n \leq X} \sum_{\substack{p_1 \neq p_2 \\ p_1, p_2 \in S_\mathcal{P}(Z)}} (p_1 p_2)^{1/2} r(n, p_1) r(n, p_2) 
	\ll 
	X (Z / \log Z)^2.
	\end{equation*}
\end{prop}

\vspace{0em}
\subsection{General strategy}

For $p_1, p_2 \in S_\mathcal{P}(Z)$ and $p_1 \neq p_2$, we will study the Dirichlet series
\begin{equation}\label{Diri}
	\sum_{n = 1}^\infty r(n, p_1)r(n, p_2)  n^{-s},
\end{equation}
and we will then use Perron's formula to estimate the partial sum
\begin{equation}\label{sum1}
	\sum_{n \leq X} r(n, p_1)r(n, p_2).
\end{equation} 

A standard way to include the remaining $(p_1 p_2)^{1/2}$ factor is by partial summation. We will instead obtain the necessary estimates directly from the class number formula (see Proposition \ref{partialsum}). In order to apply the Perron's formula, we need to investigate analytic properties of the Dirichlet series (\ref{Diri}). A possible method is to compare it to other well-known Dirichlet series. We require the forms to be both diagonal and positive definite because there is a simple representation (\ref{preortho}) of $r(n, z_j)$ by the orthogonality relations. 

We will now establish the following proposition, which is slightly more general than what we need for Proposition \ref{positive}. 
We denote
$$\mathcal{W} = \{z \in \NN: z \equiv 1 \, (\mod 4) \text{ and square-free} \} \supset \mathcal{P}.$$
Let $\Delta$ denote the set of discriminants and $\Delta_0$ denote the set of fundamental discriminants. 
For $d \in \Delta$, let $\chi_d$ be the corresponding quadratic character, specifically $\chi_d(n) = (\frac{d}{n})$. 

\begin{prop}\label{partialsum}
Let $Z \ll X^{1/10}$. Suppose that $z_1, z_2 \in S_\mathcal{W}(Z)$ and $z_1 \neq z_2$. For $j = 1, 2$, we denote $z_j^* = -4z_j$. Then we have
\begin{eqnarray*}
\nonumber \sum_{n \leq X} r(n, z_1)r(n, z_2) & = & \frac{\pi^2 X}{\sqrt{z_1^* z_2^*}} \cdot 
    \frac{L(1, \chi_{z_1^*}\chi_{z_2^*})}{L(2, \chi_{z_1^*}\chi_{z_2^*})} \cdot
    \sum_{\substack{d \in \Delta_0 \\ d| (z_1^*, z_2^*)}} \prod_{p|d} \, \frac{1-p^{-1}}{1-\chi_{z_1^*/d}\chi_{z_2^*/d}(p)p^{-1}} \\
     & & \ \ \ + \ O (X^{3/4 + \epsilon} \cdot Z),
\end{eqnarray*}

\vspace{0.25em}
\noindent
where the implicit constant is independent to the choice of $z_1, z_2$.
\end{prop}

Let $\mathcal{C}_j$ be the class group of $\QQ(\sqrt{z_j^*})$ and $h_j$ be the class number; let $\chi_j \in \widehat{\mathcal{C}_j}$ be a character. We denote 
\begin{equation}\label{def-a}
	a_j(n, \chi_j) = \sum_{\mathcal{N}\mathfrak{a} = n} \chi_j(\mathfrak{a}).
\end{equation}
Note that for $\mathrm{Re(s)} > 1$, we have
\begin{equation}\label{cgL}
	L_{\QQ(\sqrt{z_j^*})}(s, \chi_j)
	=
	\sum_{n = 1}^\infty a_j(n, \chi_j)n^{-s},
\end{equation}
where $L_{\QQ(\sqrt{z_j^*})}(s, \chi_j)$ is the class group $L$-function.

By the orthogonality relations, we have 
\begin{equation}\label{preortho}
	r(n, z_j) 
	= \frac{1}{h_j} \sum_{\mathcal{N}\mathfrak{a} = n} \sum_{\chi_j \in \widehat{\mathcal{C}_j}} \chi_j(\mathfrak{a})
	= \frac{1}{h_j} \sum_{\chi_j \in \widehat{\mathcal{C}_j}} 
	a_j(n, \chi_j). 
\end{equation}
Hence we have
\begin{equation}\label{ortho}
	r(n, z_1) r(n, z_2) = \frac{1}{h_1 h_2} \sum_{\chi_1 \in \widehat{\mathcal{C}_1}} \sum_{\chi_2 \in \widehat{\mathcal{C}_2}}
	a_1(n, \chi_1) a_2(n, \chi_2). 
\end{equation}
This makes it possible for us to study the Dirichlet series (\ref{Diri}) by investigating the Dirichlet series
\begin{equation}\label{ADiri}
	\sum_{n = 1}^{\infty} a_1(n, \chi_1) a_2(n, \chi_2) n^{-s}.
\end{equation}

The function $a(n, \chi)$ is multiplicative (we may temporarily forget the index), so it is crucial to detect its values on prime powers. 
By direct computation, we obtain the following lemma.

\begin{lemma}\label{isr}
Let $z \equiv 1 \, (\mod 4)$ be positive and square-free. We denote $z^* = -4z$ and $A = \mathcal{O}_{\QQ(\sqrt{z^*})}$. 
\begin{itemize}
	\item If $\chi_{z^*}(p) = -1$, then $(p)_A$ is a prime ideal with norm $p^2$, and
        \begin{equation*}\label{inert}
	       a(p^k, \chi) = 
	      \frac{(-1)^k + 1}{2}.
        \end{equation*}
    
    \vspace{0.1em}    
    \item If $\chi_{z^*}(p) = 1$, then $(p)_A = \pri \overline{\pri}$, where $\pri$ and $\overline{\pri}$ are both primes ideals with norm $p$, and
        \begin{equation*}\label{splits}
	             a(p^k, \chi) = \sum_{j = 0}^k \chi(\pri)^j \chi(\overline{\pri})^{k-j}.
        \end{equation*}
        
   \item If $\chi_{z^*}(p) = 0$, then $(p)_A = \pri^2$, where $\pri$ is a prime ideal with norm $p$, and
         \begin{equation*}\label{ramify}
	         a(p^k, \chi) = \chi(\pri)^k.
         \end{equation*}
\end{itemize}
In all case, we have
$$|a(p^k, \chi)| \leq k+1.$$
\end{lemma}

\begin{remark}
It is worth mentioning that recently there has been notable activity in investigating representation of integers by binary form, for example the work of \cite{LSX}, \cite{SX} and \cite{FW}. The latter is slightly similar to the problem studied in \cite{HV} and this paper.	
\end{remark}

\vspace{-0.5em}
\begin{notation}

For $s \in \CC$, we denote $\sigma = \mathrm{Re}(s)$ and $t = \mathrm{Im}(s)$.

For a positive square-free integer $z \equiv 1 \, (\mod 4)$, we will continue to write $z^* = - 4 z$.

For $c \in \RR$, we denote $\int_{(c)}$ to be the vertical complex integral $\int_{c - i\infty}^{c + i\infty}$, and $\int_{(c)_T}$ to be the truncated vertical complex integral $\int_{c - iT}^{c + iT}$.

We will continue use the symbol $\Delta$ to denote the set of discriminants and $\Delta_0$ to denote the set of fundamental discriminants. 
\end{notation}

\vspace{-0.5em}
\begin{Acknowledgements}
    This article is a version the author's master thesis at the University of Bonn.
    The author would like to thank his  advisor Valentin Blomer for introducing the problem, and giving generous feedback and encouragement along the way, especially during the global pandemic.
    The author thanks Edgar Assing for his lectures on analytic number theory.
    Finally, the author is grateful to the anonymous referees for their valuable time and comments.
\end{Acknowledgements}

\vspace{1em}
\section{Analytic properties of the Dirichlet series (\ref{ADiri})}

The main result of this section is the following proposition.
\begin{prop}\label{real}
The Dirichlet series \emph{(\ref{ADiri})}
is absolutely convergent for $\sigma > 1$. It has a meromorphic continuation in the half plane $\sigma > 1/2$, and the only possible (simple) pole is at $s = 1$, which appears only if $\chi_1$ and $\chi_2$ are both real characters. 
\end{prop}
The first statement is clear, since by Lemma \ref{isr} we have 
\begin{equation}\label{ap}
	|a_1(p^k, \chi_1) a_2(p^k, \chi_2)| \leq (k+1)^2.
\end{equation}

For the remaining parts, we will consider a class group $L$-function on the field $K = \QQ(\sqrt{z_1^*}, \sqrt{z_2^*})$, with a character $\chi$ defined first on prime ideals by
\begin{equation}\label{char}
	\chi(\Pri) = \chi_1(\mathcal{N}_1 \Pri) \cdot \chi_2(\mathcal{N}_2 \Pri).
\end{equation}
Here $\mathcal{N}_j$ denotes the ideal norm from $\mathcal{O}_K$ to $\mathcal{O}_{\QQ(\sqrt{z_j^*})}$, and we extend $\chi$ to all non-zero fractional ideals by multiplicativity. Note that the ideal norm preserves principal ideals, so it indeed induces a character on the class group of $K$. 

\vspace{0.2em}
Now we consider the following two Euler products:
\begin{equation*}\label{Euler-a}
	\sum_{n = 1}^{\infty} 
a_1(n, \chi_1) a_2(n, \chi_2) n^{-s}
 = 
 \prod_p A_p(s, \chi_1, \chi_2)
 =
 \prod_{p}  \big( 1 + \sum_{k = 1}^\infty \alpha_k(p, \chi_1, \chi_2)p^{-ks} \big),
\end{equation*}
and
\begin{equation*}\label{Euler-L}
	L_K(s, \chi) 
	=  
	\prod_p B_p(s, \chi_1, \chi_2)
	=
	\prod_p \big( 1 + \sum_{k = 1}^\infty \beta_k(p, \chi_1, \chi_2)p^{-ks} \big).
\end{equation*}
Note that $L_K(s, \chi)$ is a degree $4$ $L$-function with unit local parameters, so we have 
\begin{equation}\label{bp}
	|\beta_k(p, \chi_1, \chi_2)|
	\leq \binom{k+3}{3} \leq (k+1)^3.
\end{equation}

Let $p$ be an unramified prime, namely $p \nmid z_1^*z_2^*$. Then by Lemma \ref{isr} and (\ref{char}), we know that $\alpha_1(p, \chi_1, \chi_2) \neq 0$ if and only if $\beta_1(p, \chi_1, \chi_2) \neq 0$, and they are both equivalent to $\chi_{z_1^*}(p) = \chi_{z_2^*}(p) = 1$. In this case, we write 
\vspace{-0.5em}
$$(p) \mathcal{O}_{\QQ(\sqrt{z_1^*})} = \pri_1 \overline{\pri_1}; \ 
(p) \mathcal{O}_{\QQ(\sqrt{z_2^*})} = \pri_2 \overline{\pri_2}; \  
(p) \mathcal{O}_K = \prod_{j = 1}^4 \Pri_j. $$

\vspace{-0.2em}
\noindent
We may check that
$$\alpha_1(p, \chi_1, \chi_2) = \beta_1(p, \chi_1, \chi_2) =  (\chi_{1}(\pri_1) + \chi_{1}(\pri_1)^{-1}) \cdot (\chi_{2}(\pri_2) + \chi_{2}(\pri_2)^{-1}).$$
Hence by (\ref{ap}), we can apply Lemma \ref{L-q} (in Appendix \ref{app}) to obtain a holomorphic function $G(s, \chi_1, \chi_2)$ defined in the half plane $\sigma > 1/2$, such that
\begin{equation}\label{L-G}
	\sum_{n = 1}^{\infty} 
a_1(n, \chi_1) a_2(n, \chi_2) n^{-s} = L_K (s, \chi) \cdot G(s, \chi_1, \chi_2),
\end{equation}
for $\sigma > 1$. 

Recall that the class group $L$-function $L_K(s, \chi)$ has a meromorphic continuation in the whole complex plane; it has a pole at $s = 1$ if and only if the character $\chi$ is principal. Hence our Dirichlet series (\ref{ADiri}) has a meromorphic continuation in the half plane $\sigma > 1/2$, and we may conclude the proof by showing the following lemma. 

\begin{lemma}
	Suppose that $\chi$ is a principal character, then $\chi_1$ and $\chi_2$ are both real characters. 
\end{lemma}
\begin{proof}
We only need to consider the unramified primes here, since every ideal class contains infinite elements. Moreover, for any rational prime $p$ which is inert in both extensions, we know  $\chi_1((p)) = \chi_2((p)) = 1$ automatically. There are only two remaining situations:
\begin{itemize}
	\item For a totally split rational prime $p$, we have  
\vspace{-0.2em}
$$(p) \mathcal{O}_{\QQ(z_1^*)} = \pri_1 \overline{\pri_1}, \ (p) \mathcal{O}_{\QQ(z_1^*)} = \pri_2 \overline{\pri_2} \text{, and } (p) \mathcal{O}_K = \prod_{j = 1}^4 \Pri_j. $$ 

\vspace{-0.3em}
\noindent
Then $\chi(\Pri_j) = 1$ for every $j$ implies that

\vspace{-1.2em}
$$	\chi_1(\pri_1) \chi_2(\pri_2)
	= \chi_1(\pri_1)^{-1} \chi_2(\pri_2)
	= \chi_1(\pri_1) \chi_2(\pri_2)^{-1}
	= \chi_1(\pri_1)^{-1} \chi_2(\pri_2)^{-1}
	= 1. $$

\vspace{0.3em}
\noindent
Hence we have $\chi_1^2(\pri_1) = \chi_2^2(\pri_2) = 1$.
\vspace{0.5em}    
    \item Let $p$ be a rational prime which is split in one quadratic field and is inert in the other. Without loss of generality we assume that $\chi_{z_1^*}(p) = 1$ and $\chi_{z_2^*}(p) = -1$, then
    
\vspace{-0.8em}
$$(p) \mathcal{O}_{\QQ(z_1^*)} = \pri \overline{\pri}, \ (p) \mathcal{O}_{\QQ(z_1^*)} = (p)\text{, and } (p) \mathcal{O}_K = \Pri \overline{\Pri}.$$

\vspace{0.35em}
\noindent Now $\chi(\Pri) = 1$ implies
$\chi_1(\pri^2)\chi_2((p)) = 1$, so $\chi_1^2(\pri) = \chi_1^2(\overline{\pri}) = \chi_2((p)) = 1$. 
\end{itemize}
\end{proof}

Let $\mathcal{A}(s, \chi_1, \chi_2)$ denote the meromorphic continuation of the Dirichlet series (\ref{ADiri}) for $\sigma > 1/2$. It has the following uniform convexity bound.
\begin{lemma}\label{uni-cb}
	Let $\epsilon, \, \delta > 0$ be sufficiently small. Then in the vertical strip $\sigma \in [1/2 + \epsilon, 1]$, we have
\begin{equation*}
	\mathcal{A}(s, \chi_1, \chi_2) \ll \big((1 + |t|) \cdot Z \big)^{2(1 - \sigma) + \delta},
\end{equation*}

\vspace{0.25em}
\noindent
for any $z_1, z_2 \in S_\mathcal{W}(Z), \, z_1 \neq z_2$, and for any $\chi_1 \in \widehat{\mathcal{C}_1}, \, \chi_2 \in \widehat{\mathcal{C}_2}$.
\end{lemma} 

\begin{proof}
For any $\chi_1, \chi_2$ and for any prime $p$, by (\ref{ap}) and (\ref{bp}) we know that the quotient
$$ \frac{\, |A_p(s, \chi_1, \chi_2)| \, }{|B_p(s, \chi_1, \chi_2)|} $$ 
is uniformly bounded in the half plane $\sigma \geq 1/2 + \epsilon$.
By Lemma \ref{L-q}, we have a decomposition
\begin{equation*}
	G(s, \chi_1, \chi_2) = G_\infty(s, \chi_1, \chi_2) \cdot \prod_{p \in V} \frac{A_p(s, \chi_1, \chi_2)}{B_p(s, \chi_1, \chi_2)},
\end{equation*}
where $G_\infty$ is absolutely convergent for $\sigma > 1/2$, and $V$ is a finite set of primes, with size less than the number of ramified primes plus an absolute constant\footnote{ By Lemma \ref{L-q}, this constant only depends on $d$ and $n$ (see Appendix \ref{app} for the definitions), which are both fixed here. }. Again by (\ref{ap}) and (\ref{bp}), we know that $G_\infty$ is uniformly bounded for $\sigma \geq 1/2 + \epsilon$.
Note that for any $z_1, z_2 \in S_\mathcal{W}(Z)$, the number of ramified primes is  $O(\frac{\log Z}{\log \log Z})$.
Hence we have uniformly
\begin{equation}\label{G}
	G(s, \chi_1, \chi_2) \ll_\epsilon \exp \big( \frac{\log Z}{\log \log Z} \big),
\end{equation}
in the half plane $\sigma \geq 1/2 + \epsilon$.

The convexity bound of class group $L$-functions shows that
\begin{equation*}
	L_K(s, \chi) 
	\ll
	\big| (1 + |t|)^{\deg(K/\QQ)} \cdot \Delta_{K/\QQ}  \big|^{(1 - \sigma)/2 + \delta}
	\ll
	\big|(1 + |t|)^4 \cdot Z^4 \big|^{(1 - \sigma)/2 + \delta}.
\end{equation*}
Along with (\ref{L-G}) and (\ref{G}), we have uniformly
\begin{equation*}
	\mathcal{A}(s, \chi_1, \chi_2) \ll \big((1 + |t|) \cdot Z \big)^{2(1 - \sigma) + \delta},
\end{equation*}
for $\sigma \in [1/2 + \epsilon, 1]$.
\end{proof}

\vspace{0.5em}
Let $z \equiv 1 \, ( \mod 4) $ be positive and square-free.
Following \cite[Chapter 22.3]{IK}, the real characters on the class group of $\QQ{(\sqrt{z^*})}$ are simply the \emph{genus characters}. For $f, g \in \Delta_0$ and $f g = z^*$, the genus character $\chi_{f, g}$ is first defined on prime ideals by
\begin{eqnarray}
	\chi_{f, g}(\pri) = 
	\begin{cases}
      \chi_{f}(\mathcal{N}\pri), & \text{if $\pri \nmid f$,} \\
      \chi_{g}(\mathcal{N}\pri), & \text{if $\pri \nmid g$.}  
    \end{cases}   
\end{eqnarray}
And we extend $\chi_{f,g}$ to all non-zero fractional ideals by multiplicativity.

Along with (\ref{cgL}), the Kronecker factorization formula \cite[(22.57)]{IK}:
\begin{equation*}
	L_{\QQ(\sqrt{z^*})}(s, \chi_{f, g})
	= L(s, \chi_{f}) L(s, \chi_{g})
\end{equation*} 
implies that
\begin{equation*}
	a(n, \chi_{f,g}) = (\chi_f * \chi_g)(n).
\end{equation*}

\vspace{0.25em}
\noindent
We may expand the product of two of these convolutions.

\begin{lemma}\label{moebius}
	Suppose that $f_1, g_1, f_2, g_2 \in \Delta_0$. Then we have
\begin{equation*}
	(\chi_{f_1} * \chi_{g_1})(n) \cdot (\chi_{f_2} * \chi_{g_2})(n) 
	= \sum_{n = abcde^2} \mu(e) \chi_{f_1}(abe) \chi_{g_1}(cde) \chi_{f_2}(ace) \chi_{g_2}(bde).
\end{equation*}
\end{lemma} 

\begin{proof}
	We may write the left hand side as
	$$
	\sum_{n = rs} \chi_{f_1}(r) \chi_{g_1}(s) \cdot \sum_{n = tu} \chi_{f_2}(t) \chi_{g_2}(u).
	$$
	For fixed $r, s, t, u$, let $m$ be the greatest common divisor of them. Then for each $e | m$, there are $\tau(m/e)$ possible decompositions satisfying $abe = r, cde = s, ace = t, bde = u$.
	Then by M\"{o}bius inversion, we have
	$$
	\sum_{e | m} \mu(e) \tau(m/e) = 1,
	$$
	\vspace{-0.5em}
	
	\noindent which implies that we count each decomposition from the left hand side exactly once.
\end{proof}

We can now prove the following lemma.

\begin{lemma}\label{Afactor}
For $j = 1, 2$, let $\chi_j = \chi_{f_j, g_j}$ be a genus character of the class group of $\QQ(\sqrt{z_j^*}) $. 
Then $\mathcal{A}(s, \chi_1, \chi_2)$ has a pole at $s = 1$, if and only if there exists a fundamental discriminant $d \, | (z_1^*, z_2^*)$, such that
	\begin{equation*}
		\mathcal{A}(s, \chi_{1}, \chi_{2})  = \,
		\frac{\zeta^{(d)}(s) L(s, \chi_{z_1^*}) L(s, \chi_{z_2^*}) L(s, \chi_{z_1^*/d}\chi_{z_2^*/d})}{L(2s, \chi_{z_1^*}\chi_{z_2^*})},
	\end{equation*}
	where $\zeta^{(d)}(s) = \zeta(s) \cdot \prod_{p|d} (1-p^{-s})$.
\end{lemma}
\begin{proof}
By Lemma \ref{moebius}, we have
\vspace{-0.25em}
    \begin{eqnarray}\label{5convo}
      \nonumber \mathcal{A}(s, \chi_1, \chi_2) \ =  & \displaystyle \sum_{n = 1}^{\infty} \ \sum_{n = abcde^2} \mu(e) \chi_{f_1}(abe) \chi_{g_1}(cde) \chi_{f_2}(ace) \chi_{g_2}(bde) \, n^{-s} &  \\
   = & \displaystyle \sum_{n = 1}^{\infty} \  \chi_{f_1f_2}(a) \, \chi_{f_1g_2}(b) \, \chi_{g_1f_2}(c) \, \chi_{g_1g_2}(d) \, \mu(e) \chi_{z_1^*z_2^*} (e) \, n^{-s} &   \\ 
      \nonumber  = & L(s, \chi_{f_1f_2}) L(s, \chi_{f_1g_2}) L(s, \chi_{g_1f_2}) L(s, \chi_{g_1g_2}) L(2s, \chi_{z_1^*z_2^*})^{-1}. & 
    \end{eqnarray} 
\vspace{-0.75em}

\noindent 
Note that $f_1, f_2, g_1, g_2 \in \Delta_0$, and $z_1, z_2$ are square-free. The function $\mathcal{A}(s, \chi_1, \chi_2)$ has a pole at $s = 1$, if and only if $f_1 = f_2$, or $f_1 = g_2$, or $f_2 = g_1$ or $g_1 = g_2$. Without loss of generality, we may assume that $f_1 = f_2 = d$. In this case, the factorization in (\ref{5convo}) becomes
\begin{equation*}
\frac{\zeta^{(d)}(s) L(s, \chi_{z_1^*}) L(s, \chi_{z_2^*}) L(s, \chi_{z_1^*/d}\chi_{z_2^*/d})}{L(2s, \chi_{z_1^*}\chi_{z_2^*})}.
\end{equation*}
\end{proof}

\vspace{1em}
\section{Proof of Proposition \ref{partialsum}}

Let $\alpha(n)$ be an arithmetic function. We denote its Dirichlet series by 
\vspace{-0.3em}
\begin{equation*}
	\mathcal{D}(s) = \sum_{n = 1}^{\infty} \alpha(n) n^{-s}.
\end{equation*}

\vspace{-0.2em}
\noindent
We have the following version of Perron's formula (\cite[Lemma 1.4.2]{Br}).
\begin{theorem}[Perron's formula, an effective version]\label{Perron}
Let $c > 0$ and $X, T \geq 2$. Suppose that $\mathcal{D}(s)$ is absolutely convergent in the half plane $\sigma \geq c$.
Then we have
    \begin{eqnarray*}
        	\sum_{n \leq X} \alpha(n) 
        	 =  \frac{1}{2\pi i} \int_{(c)_T} \mathcal{D}(s) X^s \frac{ds}{s}  +  O\Big(\frac{X^c}{T} \sum_{n = 1}^{\infty} |\alpha(n)|n^{-c} + A_X\big(1 + \frac{X \log X}{T}\big)\Big), 
    \end{eqnarray*}
    where $A_X = \max_{\frac{3}{4}X \leq n \leq \frac{5}{4}X} |\alpha(n)|$.

\end{theorem}
Let $\alpha(n) = r(n, z_1) r(n, z_2)$.
By (\ref{ortho}) and Proposition \ref{real}, we know that 
\begin{equation}\label{Ds}
   \mathcal{D}(s) = \sum_{n = 1}^{\infty} r(n, z_1) r(n, z_2) n^{-s}
   = \frac{1}{h_1 h_2} \ \sum_{\substack{\chi_1, \chi_2}} \ \sum_{n = 1}^\infty a_1(n, \chi_1) a_2(n, \chi_2) n^{-s} .
\end{equation}
has a meromorphic continuation for $\sigma > 1/2$. 
For $\delta > 0$ sufficiently small, by (\ref{ortho}) and (\ref{ap}) we have $|\alpha(n)| \ll n^\delta$, and hence $A_X \ll X^\delta$. Let $\epsilon = \delta / 10$, then by Theorem \ref{Perron} we have
\begin{eqnarray}\label{3.2}
\sum_{n \leq X} r(n, z_1) r(n, z_2) 
        	  =  \frac{1}{2\pi i} \int_{(1 + \epsilon)_T} \mathcal{D}(s) X^s \frac{ds}{s}  +  O \Big(\frac{X^{1 + \epsilon}}{T} + X^{\delta} \Big). 
\end{eqnarray}

In the region $\sigma \in [1/2+\epsilon, 1+\epsilon], \, t \in [-T, T]$, we know that
$\mathcal{D}(s) X^s s^{-1}$
has only one pole at $s = 1$. 
So by the residue theorem, the main term of (\ref{3.2}) becomes
\begin{eqnarray}\label{contour}
	X \cdot \underset{s = 1}{\res} \mathcal{D}(s) 
	+ \frac{1}{2 \pi i}
	\left( \int_{\frac{1}{2} + \epsilon + iT}^{1 + \epsilon + i T} 
	- \int_{\frac{1}{2} + \epsilon - iT}^{1 + \epsilon - i T} 
	+ \int_{(\frac{1}{2} + \epsilon)_T}
	\right) \, \mathcal{D}(s) \cdot \frac{X^s}{s} ds.
\end{eqnarray}
For those $s$ on these integral segments, by (\ref{ortho}) and Lemma \ref{uni-cb} we have
\begin{equation*}
	\mathcal{D}(s) \ll \big|(1+|t|)\cdot Z \big|^{\max(0, \, 2(1 - \sigma) + \delta)}.
\end{equation*}
If we assume that $X, T$ sufficiently large and $T \leq \sqrt{X}/2$, then we have
\begin{eqnarray}
  \int_{\frac{1}{2} + \epsilon + iT}^{1 +\epsilon + i T} 
    & \ll & Z \cdot \int_{\frac{1}{2} + \epsilon}^{1 + \epsilon} T^{\max(-1, \, 1 - 2\sigma + \delta)} \cdot X^{\sigma} d \sigma \\
\nonumber   & \ll &  Z \cdot (X^{1/2+\epsilon} + X^{1 + \epsilon} \cdot T^{-1}). 
\end{eqnarray}
The estimate for the segment $[\frac{1}{2} + \epsilon - iT, 1 + \epsilon -iT]$ is similar. For the vertical integral, we have
\begin{eqnarray}
 \int_{(\frac{1}{2} + \epsilon)_T}
    & \ll & Z \cdot \int_{-T}^T (1 + |t|)^{1 - 2\epsilon + \delta} \cdot \frac{X^{1/2 + \epsilon}}{|\frac{1}{2} + \epsilon + it|} dt \\
\nonumber   & \ll & Z \cdot X^{1/2 + \epsilon} \cdot T^{1 - 2\epsilon +\delta}. 
\end{eqnarray}
Now we choose $T = X^{1/4}$, then (\ref{3.2}) becomes
\begin{equation}\label{rDO}
	\sum_{n \leq X} r(n, z_1) r(n, z_2) = X \cdot \underset{s = 1}{\res} \mathcal{D}(s) + O(X^{3/4 + \epsilon} \cdot Z).    
\end{equation}
Since the convexity bound (Lemma \ref{uni-cb}) is uniform for all $z_1, z_2 \in S_\mathcal{W}(Z)$, the implicit constant here is independent to the choice of $z_1$ and $z_2$.

By (\ref{Ds}), Proposition \ref{real} and Lemma \ref{Afactor}, we have
\begin{equation}\label{sumfact}
	\mathcal{D}(s) = \frac{1}{h_1h_2} \sum_{\substack{ d \in \Delta_0 \\ d| (z_1^*, z_2^*)}} \frac{\zeta^{(d)}(s) L(s, \chi_{z_1^*}) L(s, \chi_{z_2^*}) L(s, \chi_{z_1^*/d}\chi_{z_2^*/d})}{L(2s, \chi_{z_1^*}\chi_{z_2^*})} + H(s).
\end{equation}
Here $H(s)$ is the contribution of the pair of characters such that the Dirichlet series (\ref{ADiri}) has no pole at $s = 1$. Recall that the class number formula shows that
\begin{equation*}
	h_j = \frac{\sqrt{4z_j}}{\pi} L(1, \chi_{z_j^*}).
\end{equation*}
Hence we obtain
\begin{equation*}
	X \cdot \underset{s = 1}{\res} \mathcal{D}(s)
	= X \cdot \frac{\pi^2}{\sqrt{z_1^* z_2^*}} \cdot 
    \frac{L(1, \chi_{z_1^*}\chi_{z_2^*})}{L(2, \chi_{z_1^*}\chi_{z_2^*})} \cdot
    \sum_{\substack{d \in \Delta_0 \\ d| (z_1^*, z_2^*)}} \prod_{p|d} \frac{1-p^{-1}}{1-\chi_{z_1^*/d}\chi_{z_2^*/d}(p)p^{-1}}\,.
\end{equation*}
Along with (\ref{rDO}), this concludes the proof of Proposition \ref{partialsum}.

\vspace{1em}
\section{Proof of Proposition \ref{positive}}

According to Proposition \ref{partialsum}, we have
\begin{eqnarray}\label{step1-p}
\nonumber \displaystyle \sum_{n \leq X} \sum_{\substack{p_1 \neq p_2 \\ p_1, \, p_2 \in S_\mathcal{P}(Z) }} (p_1 p_2)^{1/2} r(n, p_1) r(n, p_2)  \ \ll \ X \cdot  \sum_{\substack{p_1 \neq p_2 \\ p_1, \,  p_2 \in S_\mathcal{P}(Z)}}   L(1, \chi_{p_1^*}\chi_{p_2^*}).   
\end{eqnarray}
Then it is sufficient to prove Proposition \ref{positive} by showing the following lemma.

\begin{lemma}
	We have
	\begin{equation*}
       \sum_{\substack{p_1 \neq p_2 \\ p_1, \,  p_2 \in S_\mathcal{P}(Z)}}   L(1, \chi_{p_1^*}\chi_{p_2^*}) \ll Z^2 \log^{-2} Z.
    \end{equation*}
\end{lemma}
	
\begin{proof}

Let $\chi$ be a primitive character modulo $q > 1$. Then by the approximate functional equation for Dirichlet $L$-functions \cite[Corollary (a)]{Ra}, we have
\begin{eqnarray*}
	L(1, \chi) = \sum_{n \leq Y} n^{-1} \chi(n) + O \Big( q^{1/2}Y^{-1} \log(q+2) \Big).
\end{eqnarray*} 
We choose $Y = Z \log Z$ and sum over $p_1, p_2$ to obtain
\begin{eqnarray*}\label{L-avg-p}
	\sum_{\substack{p_1 \neq p_2 \\ p_1, \, p_2 \in S_\mathcal{P}(Z) }}   L(1, \chi_{p_1^*}\chi_{p_2^*})	\ = \  
	\sum_{\substack{n \leq Z \log Z \\[+1pt] n \text{ odd}  }} \ n^{-1}  \sum_{p_1, \, p_2} \chi_{p_1p_2}(n) \, + \,  O ( Z^2 \log^{-2} Z ).
\end{eqnarray*}
Recall $p^* = -4p$, so we only need to sum over odd $n$ and we have $\chi_{p_1^*} \chi_{p_2^*}(n)$ = $\chi_{p_1 p_2}(n)$. 

For the sum over squares, we have
\begin{equation*}
	\sum_{n = \square} \ n^{-1} \sum_{p_1, \, p_2}  \chi_{p_1p_2}(n)
	\ll Z^2 \log^{-2} Z. 
\end{equation*}
For the remaining sum, note that
\begin{equation*}
	\sum_{\substack{n \leq Z \log Z \\ n \text{ odd}, \, n \neq \square}} \ n^{-1} \, 
 \sum_{\substack{p_1 \neq p_2 \\ p_1, \, p_2 \in S_\mathcal{P}(Z) }} \chi_{p_1p_2}(n) \,
 =
 \sum_{\substack{n \leq Z \log Z \\ n \text{ odd}, \, n \neq \square}} \ n^{-1} \, 
 \Big| \sum_{\substack{p \in S_\mathcal{P}(Z)}} \chi_{p}(n) \Big|^2 + O(Z). 
\end{equation*}
Hence it is sufficient to show that
\begin{equation*}
	 \sum_{\substack{n \leq Z \log Z \\ n \text{ odd}, \, n \neq \square}} \ n^{-1} \, 
 \Big| \sum_{\substack{p \in S_\mathcal{P}(Z)}} \chi_{p}(n) \Big|^2 
 \ll Z^2 \log^{-2} Z. 
\end{equation*}

We denote
\begin{equation*}
	A(n,Z) = n^{-1} \, 
 \Big| \sum_{\substack{p \in S_\mathcal{P}(Z)}} \chi_{p}(n) \Big|^2.
\end{equation*}
For a positive odd non-square integer n, we decompose $n = n_1 n_2^2$ with $n_1$ square-free. Observe that
\begin{eqnarray*}
   A(n, Z) \,
   \leq  \,
   n_1^{-1}n_2^{-2} \
   \Big| \sum_{\substack{p \in S_\mathcal{P}(Z)}} \chi_{p}(n_1) + O\big(\tau(n_2)\big) \Big|^2.
\end{eqnarray*}
Then we have
\begin{eqnarray*}
   \sum_{\substack{n \leq Z \log Z \\ n \text{ odd}, \, n \neq \square}} A(n,Z) \
   = \
   \sum_{\substack{n_2^2 \leq Z \log Z \\ n_2 \text{ odd}}} \ n_2^{-2} \
   \Big( \,
   \sump_{n_1 \leq n_2^{-2} Z \log Z} \, A(n_1,Z) + O(Z^{1+\epsilon}) \Big),
\end{eqnarray*}
where $\sum^*$ henceforth indicates restriction to positive odd square-free integers. Now we may conclude the proof by showing that
\begin{equation}\label{cs-p}
	 \sump_{\substack{n \leq Z \log Z}}  \ A(n,Z) 
 \ll Z^2 \log^{-2} Z. 
\end{equation}

\vspace{0.2em}
According to Heath-Brown's quadratic large sieve \cite[Theorem 1]{HB}, we have
\begin{equation*}
	\sump_{\substack{n \leq Z \log Z}} \ \Big| \sum_{p \in S_\mathcal{P}(Z)} \chi_p(n) \Big|^2 \ll Z^{2+\epsilon}.
\end{equation*}
Then by partial summation, we obtain
\begin{eqnarray*}
	\sump_{\substack{n \leq Z \log Z}} A(n, Z)
	= 
	\sum_{N < Z \log Z} \ N^{-1}(N+1)^{-1} \cdot \sump_{\substack{n \leq N}} \, \Big| \sum_{p \in S_\mathcal{P}(Z)} \chi_p(n) \Big|^2
	+ O(Z^{1 + \epsilon}).
\end{eqnarray*}
Now we separate the $N$-sum into the following three sums:
\begin{equation*}
	S_1 = \sum_{N \leq \log^B Z}, 
	\ \
	S_2 = \sum_{\log^B Z < N \leq Z^\delta},
	\ \
	S_3 = \sum_{Z^\delta < N < Z \log Z}.
\end{equation*}
The parameters $B$ and $\delta$ will be determined later. We will obtain the required upper bound of $S_1$ by applying the Siegel-Walfisz theorem, and of $S_2$ by using Heath-Brown's quadratic large sieve once more. For the intermediate sum $S_2$, we will apply a lemma by Elliott.

\vspace{0.5em}
In fact, we have
\begin{eqnarray*}
	\sum_{p \in S_\mathcal{P}(Z)} \chi_p(n) 
	= 
	\sum_{\substack{d \, (\mod 4n) \\ d \equiv 1 \,(\mod 4)}} \chi_d(4n) \cdot \pi(Z; 4n,d),
\end{eqnarray*}
where
\begin{eqnarray*}
	\pi(Z; 4n, d) : = \# \{ p \leq Z: p  \equiv d \,(\mod 4n)\}.
\end{eqnarray*}
Since $n$ is odd, by the Chinese remainder theorem we know
\begin{eqnarray*}
\sum_{\substack{d \, (\mod 4n) \\ d \equiv 1 \,(\mod 4)}} \chi_d(4n) 
=
	\sum_{\substack{d \, (\mod 4n) \\ d \equiv 1 \,(\mod 4)}} \chi_d(n)
	= \sum_{\substack{d' \, (\mod n)}} \left( \frac{d'}{n} \right).
\end{eqnarray*}
For a positive odd square-free integer $n$, note that $\left( \frac{\cdot}{n} \right)$ is a non-principal Dirichlet character modulo $n$.
Then by the orthogonality relations, we have
\begin{eqnarray*}
	\sum_{\substack{d \, (\mod 4n) \\ d \equiv 1 \,(\mod 4)}} \chi_d(4n)
	= 0.
\end{eqnarray*}
Now choose $A > 2$. By \cite[Corollary 5.29]{IK}, we know that\footnote{ When $n \leq \log(Z)^{A+1}$, it is the classsical Siegel-Walfisz theorem. For larger $n$, this display is simply trivial.}
\begin{eqnarray*}
	\pi(Z; n,d) = \varphi(4n)^{-1} Z + O(Z \log^{-A}Z),
\end{eqnarray*}
where the implicit constant only depends on $A$. Then we have
\begin{equation*}
	\sum_{p \in S_\mathcal{P}(Z)} \chi_p(n) = \sum_{\substack{d \, (\mod 4n) \\ d \equiv 1 \,(\mod 4)}} \chi_d(4n) \cdot \pi(Z; 4n,d)
	  \ll \ n \cdot Z \log^{-A} Z.
\end{equation*}
It follows that
\begin{equation}\label{S1}
	S_1 \ll \sum_{N \leq \log^B Z} N Z^2 \log^{-2A}Z
	\ll 
	 Z^2 \log^{2B-2A}Z.
\end{equation}

After expanding the square and exchanging sums, we have
\begin{equation*}\label{sq-nsq}
	\sump_{\substack{n \leq N}} \Big| \sum_{p \in S_\mathcal{P}(Z)} \chi_p(n) \Big|^2 \,
	= \
	\Big| \sum_{\substack{p_1 \neq p_2 \\ p_1, \, p_2 \in S_\mathcal{P}(Z) }} \ \sump_{\substack{n \leq N}}  \chi_{p_1p_2}(n) \Big|
	+ O(N Z/\log Z).
\end{equation*}
An estimate by Elliott \cite[Lemma 10]{El} (or more precisely, \cite[(6)]{HB}) shows that\footnote{ Elliott's estimate is much weaker than Heath-Brown's quadratic large sieve in the $N$-aspect, but it wins an extra $Z^\epsilon$. This is exactly what we need for estimating $S_2$.}
\begin{equation*}
	\sum_{\substack{p_1 \neq p_2 \\ p_1, \, p_2 \in S_\mathcal{P}(Z) }} \Big|\sump_{\substack{n \leq N}}  \chi_{p_1p_2}(n) \Big|^2
	=
	\sum_{\substack{p_1 \neq p_2 \\ p_1, \, p_2 \in S_\mathcal{P}(Z) }} \Big|\sump_{\substack{n \leq N}} \left( \frac{n}{p_1p_2} \right) \Big|^2
	\ll (Z^2 + N^2 \log N) N.
\end{equation*}
The first equality follows from the quadratic reciprocity, since $p_1 p_2 \equiv 1 \, (\mod 4)$.
Let $N \ll Z^\delta$, where $\delta > 0$ is sufficiently small. Then by the Cauchy-Schwarz inequality, we have
\begin{equation*}
	\sum_{\substack{p_1 \neq p_2 \\ p_1, \, p_2 \in S_\mathcal{P}(Z) }} \ \sump_{\substack{n \leq N}}  \chi_{p_1p_2}(n)
	\ll N^{1/2} Z^2 \log^{-1} Z.
\end{equation*}
Hence we obtain
\begin{equation}\label{S2}
   S_2 
   \ll 
   \sum_{\log^B Z < N \leq Z^\delta} \, N^{-3/2} Z^2 \log^{-1} Z 
   \ll Z^2 \log^{-B/2 - 1}Z.
\end{equation}

Again by \cite[Theorem 1]{HB}, we have 
\begin{equation}\label{S3}
	S_3 
	\ll 
	\sum_{Z^\delta < N < Z \log Z}
	N^{-2} Z^{2 + \epsilon}
	\ll Z^{2 + \epsilon - \delta}.
\end{equation}

Now we choose $\epsilon > 0$ sufficiently small, $\delta = 2\epsilon$, $B = 2$, and $A = 3$. Then by (\ref{S1}), (\ref{S2}) and (\ref{S3}), we finally obtain (\ref{cs-p}).

\end{proof}

\vspace{1em}
\appendix

\section{Quotient of a Dirichlet series by an $L$-function}\label{app}

Let $a(n)$ be a multiplicative arithmetic function. We denote $a_k(p) = a(p^k)$. Suppose that for any prime $p$, we have 
\begin{equation}\label{bound-a}
   a_k(p) \ll k^c,
\end{equation}
where $c$ is a positive constant.
Then the Euler product
$$A(s) = \prod_{p} A_p(s) = \prod_p \big( 1 + \sum_{k = 1}^\infty a_k(p) p^{-ks} \big)
$$
is absolutely convergent for $\sigma > 1$.

Let $B(s)$ be an $L$-function of degree $d$ with unit local parameters $\lambda_k(p)$. Then the Euler product
\begin{eqnarray}\label{Lsf}
	B(s) 
	\ = \  
	\prod_p B_p(s) 
	& = &
	\prod_p \ (1 - \lambda_1(p)p^{-s})^{-1} \cdots (1 - \lambda_d(p)p^{-s})^{-1} \\
\nonumber	& = & 
	\prod_p \big( 1 + \sum_{k = 1}^{\infty} b_k(p) p^{-ks} \big)
\end{eqnarray}
is absolutely convergent for $\sigma > 1$. Note that for any prime $p$, we have uniformly
\begin{equation}\label{bound-b}
	|b_k(p)| \leq \binom{k+d-1}{d-1} \leq (k+1)^d.
\end{equation}

We will prove the following lemma.
\begin{lemma}\label{L-q}
   Let $A(s)$ and $B(s)$ as above; let $U$ be a finite set of primes and let $n \geq 2$. Suppose that for all primes $p \notin U$, we have $a_k(p) = b_k(p)$ for any $1 \leq k < n$.
   Then there exists a function $G(s)$, such that 
  $$A(s) = B(s) \cdot G(s),$$
  for $\sigma > 1$. Moreover, the function $G(s)$ is holomorphic in the half  plane $\sigma > 1/n$, and it has an Euler product with decomposition
  $$G(s) = G_\infty(s) \cdot \prod_{p \in V} \frac{A_p(s)}{B_p(s)}.$$
  Here $G_\infty(s)$ is an absolutely convergent infinite product for $\sigma > 1/n$, and $V$ is a finite set of primes with $\# V \leq \#U + C$, where $C$ is a constant only depending on $n$ and $d$.

\end{lemma}

\begin{proof} 
	We define $g_1(p) =  a_1(p) - b_1(p)$; for $k \geq 2$, we recursively define
	$$g_k(p) = a_k(p) - b_k(p) - \sum_{m = 1}^{k - 1} b_m(p) g_{k - m}(p).$$
	Then we define 
	$$G(s) = \prod_{p} G_p(s) = \prod_{p} \big(1 + \sum_{k = 1}^\infty g_k(p)p^{-ks} \big).$$ 
	For all $p \notin U$, we have $g_k(p) = 0$ for $1 \leq k < n$ . By induction, we have $g_k(p) \ll T^k$, where $T$ is a constant only depending on the degree $d$ of the $L$-function. 
	In fact, we can choose any $T > 1$ satisfying
	$$
	\sum_{m = 1}^{\infty} (m+1)^d \cdot T^{-m} < 1.
	$$
		
	For $\sigma > 1/n$, by (\ref{bound-a}) and (\ref{bound-b}) we know that the $p$-factors $A_p(s)$ and $B_p(s)$ are absolutely convergent. Moreover, for any prime $p$, by (\ref{Lsf}) we have $|B_p(s)| \geq \delta > 0$. We denote 
	\begin{equation*}\label{V}
			V = \{p \text{ prime}: p \leq (2T)^{n} \text{ or } p \in U\}.
	\end{equation*}
	Then the finite product
	\begin{equation*}\label{G0}
		\prod_{p \in V} G_p(s) =  \prod_{p \in V}\frac{A_p(s)}{B_p(s)}
	\end{equation*}	
	is holomorphic.
	For $\sigma > 1/n$ and for any prime $p \notin V$, we have
	\vspace{-0.3em}
    $$ \sum_{k = 1}^{\infty}g_k(p)p^{-ks} 
    =  \sum_{k = n}^{\infty}g_k(p)p^{-ks} 
    \ll \sum_{k = n}^{\infty}(T p^{-s})^k
    \ll \big( T p^{-s} \big)^n.$$
	Therefore, the infinite product
	\begin{equation*}\label{G1}
	   G_\infty(s) = \prod_{p \notin V} G_p(s)
	\end{equation*}
    is absolutely convergent for $\sigma > 1/n$, hence the function $G(s)$ is holomorphic for $\sigma > 1/n$. 

    \vspace{0.1em}
    For $\sigma > 1$, we have $A(s) = B(s) \cdot G(s)$, simply by comparing their coefficients.
  
\end{proof}

\end{document}